\documentclass[11pt,reqno]{article}
\usepackage{latexsym, amsmath, amssymb, amsthm, a4, epsfig}
\usepackage{graphicx}
\usepackage{color}

\newtheorem{theorem}{Theorem}[section]

\newtheorem{prop}[theorem]{Proposition}

\newcommand{\p}{\partial}

\newcommand{\Nbb}{\mathbb{N}}

\newcommand{\Rbb}{\mathbb{R}}

\newcommand{\GO}{\Omega}

\newcommand{\beq}{\begin{equation}}
\newcommand{\eeq}{\end{equation}}

\numberwithin{equation}{section}
\numberwithin{figure}{section}

\begin{document}

\title{A remark on decay rates of odd partitions: An application of spectral asymptotics of the Neumann--Poincar\'e operators   \thanks{\footnotesize This work was supported by  JSPS (of Japan) KAKENHI Grant Number JP21K13805.}}

\author{ Yoshihisa Miyanishi\thanks{Department of Mathematical Sciences,  Faculty of Science, Shinshu University,  Matsumoto 390-8621, Japan. Email: {\tt miyanishi@shinshu-u.ac.jp}.}}
%\date{2023/5/3}
\maketitle

\begin{abstract} We introduce a theorem currently proved 
unique by the asymptotic behaviors of eigenvalues of a compact operator. 
Specifically, a problem of partitions is considered and the Neumann--Poincar\'e operator is employed as the compact linear operator. 
Then a theorem is proved by the spectrum of the Neumann--Poincar\'e operator. 

Even though the proposed problem below looks artificial, our result in the partitions 
seems to be proven or relieved unique by the spectral theory of the Neumann--Poincar\'e operators: 

Odd partitions of the unit interval  $[0, 1]$ are considered, that is, we divide the unit interval $[0, 1]$ into $2N+1$ disjoint {\it non-zero} intervals $L_{N, k}$ ($k=1, \ldots, 2N+1$) and the sum of corresponding lengths $\sum_{k=1}^{2N+1} |L_{N, k}| =1$ for each $N\in \Nbb_{\geq 0}$. 

Thus we obtain a countable set of real numbers $P=\{ |L_{N, k}| \ ;\ k=1, 2, \ldots, 2N+1, \ N\in \Nbb_{\geq 0} \}$ by odd partitions of the unit interval. One can enumerate the set $P$ in decreasing order. Then the non-increasing sequence is given as   
$$
a_1=|L_{0, 1}|=1 > a_2 \geq a_3 \geq \ldots >0. 
$$
We show that for any $C \geq \frac{1}{2}$ there exist odd partitions of the unit interval such that 
$$
a_j \sim C j^{-1/2} \quad \mbox{as}\ j \rightarrow \infty. 
$$
Here the coefficient $C=1/2$ corresponds to the optimal decay. 

We prove this fact by a fundamental property of Riemann zeta function and by eigenvalue asymptotics for some  compact linear operators known as the Neumann--Poincar\'e operators.  
\end{abstract}



%\tableofcontents
%%%%%%%%%%%%%%%%%%%%%%%%%%%%%
\section{Introduction and Results}
%%%%%%%%%%%%%%%%%%%%%%%%%
In spite of the fact that many sophiscated results have been done in spectral theory of compact operators, 
their applications in number theory often seem weaker than those in specific fields. 

As a typical example, the well-known Gauss circle problem, which is the problem of determining how many integer lattice points $N(r)$  
there are in a circle centered at the origin and with radius $r>0$, it should be $N(r)=\pi r^2 +O(r^{1/2+\epsilon})$. 
Then the estimate $0< \epsilon \leq 1/2$ is proven by eigenvalue asymptotics of Lapace operator, whereas the improved estimate 
$0< \epsilon \leq 27/208$ is proven in analytic number theory \cite{DG, Ha, Hu}. 
Thus spectral theory of linear operators has brought about fortuitous results. However, when spectrtal theory is applied 
to mathematical problems of different fields, the obtained results seem often weaker than those in specific fields. 

% Our main purpose is to present a mathematical result currently proved in unique the spectral theory of a linear operator. 

Our purpose here is to present an application which seems to be currently proven unique by spectral theory. This is shown by  
the behavior of partitions. 
To be more precise, we divide the unit interval $[0, 1]$ into $2N+1$ {\it non--zero length} subintervals, then finely divided intervals seem to appear. We call such partitions ``Odd Partitions''. Denoting odd partitions as $L_{N, k}$ ($k=1. \ldots, 2N+1$) for each $N \in \Nbb_{\geq 0}$, one can enumerate the contably infinitely many real numbers $\{ |L_{N, k}| \ ;\ k=1, 2, \ldots, 2N+1, \ N\in \Nbb_{\geq 0} \}$ in decreasing order. Here $|\cdot|$ denotes the Lebesgue measure (Length). Thus such procedure allows us to give the non-increasing sequence: 
\beq
1=|L_{0,1}|=a_1 > a_2 \geq a_3 \geq \ldots >0.
\eeq 

For instance, equi-partitions of the unit interval yield a diagram: 
The first partitioning yields $3$ intervals, whose length is $1/3$. Similarly, a non-increasing sequence is produced.  
$$
\begin{array}{c}
1\vspace{1mm} \\
1/3\ \hspace{2mm} 1/3\ \hspace{2mm} 1/3 \vspace{1mm} \\
1/5\ \hspace{2mm} 1/5\ \hspace{2mm} 1/5\ \hspace{2mm} 1/5\ \hspace{2mm} 1/5  \vspace{1mm} \\
1/7\ \hspace{2mm} 1/7\ \hspace{2mm} 1/7\ \hspace{2mm} 1/7\ \hspace{2mm} 1/7\ \hspace{2mm} 1/7\ \hspace{2mm} 1/7 \vspace{1mm} \\
\vdots \qquad \vdots \qquad \vdots
\end{array}
$$
In this figure, each column shows a partition of the unit interval, that is, the sum of each column equals one. Thus we obtain the enumerated sequence in decreasing order:  
\beq\label{ex: equi-partition}
1, 1/3, 1/3, 1/3, 1/5, 1/5, 1/5, 1/5, 1/5, 1/7, 1/7, 1/7, 1/7, 1/7, 1/7, 1/7, \ldots
\eeq
After the $N$-th procedure, the diagram of partitions always consists of $N^2$ numbers. One can easily find in \eqref{ex: equi-partition} that the $j=N^2$-th number $a_j$ is $1/(2N-1) \sim \frac{1}{{2}}j^{-1/2}$ for large $j$. It is strongly expected that the optimal deacy rate is attained by such equi-partitions. 

In fact, we have the desirable decay rates: 
\begin{theorem}[Main Theorem]\label{thm: main}
For all $C\geq \frac{1}{{2}}$, there exist odd partitions of the unit interval such that 
$$
a_j \sim C j^{-1/2} \quad \mbox{as}\ j \rightarrow \infty.
$$
We here emphasize that $C=\frac{1}{{2}}$ is the minimum coefficient, 
namely, $\displaystyle\liminf_{j \rightarrow \infty} j^{1/2} a_j \geq 1/2$ for arbitrary odd  partitions. 
\end{theorem}
To prove this, we recall unconditional sums which are convenient here:  
\begin{prop}\label{prop: tau fanction}
For odd partitions, we define the infinite sum $\tau(p)$ by 
\beq\label{eq: tau function}
\tau(p) := \sum_{\substack{ N \in \Nbb_{\geq 0} \\ {k=1, 2, \ldots, 2N+1}}}|L_{N, k}|^p \quad (p>2).
\eeq
Then $\tau(p) \geq  (1-2^{1-p})\zeta(p-1)\ (p>2)$ where $\zeta(p)$ denotes Riemann zeta function. 
 The equality holds only for the case of equi-partitions. 
\end{prop}
We remark that Proposition \ref{prop: tau fanction} holds true even in the case that the sum \eqref{eq: tau function} diverges to $\infty$. 

\begin{proof}[Proof of Proposition \ref{prop: tau fanction}]
Since the sum \eqref{eq: tau function} consists only of positive values, the sum is unconditional and independent of rearrangements. 
   
It follows by H\"{o}lder's inequality  (e.g. \cite{Yo}) that 
\begin{align*}
1&= |L_{N, 1}| + |L_{N, 2}| +\cdots + |L_{N, 2N+1}| \\
&\leq  (1+ 1+ \cdots +1)^{1/q} \cdot (|L_{N, 1}|^p + |L_{N, 2}|^p +\cdots + |L_{N, 2N+1}|^p)^{1/p}  \\
&\leq (2N+1)^{\frac{p-1}{p}} \cdot (|L_{N, 1}|^p + |L_{N, 2}|^p +\cdots + |L_{N, 2N+1}|^p)^{1/p} 
\end{align*}
where $p, q \in [1, \infty]$ with $1/p + 1/q = 1$. 
Thus 
 \beq
 |L_{N, 1}|^p + |L_{N, 2}|^p +\cdots + |L_{N, 2N+1}|^p \geq (2N+1)^{(1-p)} \quad \mbox{for}\ p\geq 1
 \eeq
 and so 
 \beq
 \tau(p) \geq \sum_{N \in \Nbb_{\geq 0}} (2N+1)^{(1-p)} = (1-2^{1-p})\zeta(p-1) \quad \mbox{for}\ p>2.
 \eeq
 The equality holds only for the case of equi-partitions
 \\
  (i.e. $|L_{N, 1}| = |L_{N, 2}| = \cdots = |L_{N, 2N+1}| =1/(2N+1)$ for all $N \in \Nbb$). 
\end{proof}

%%%%%%%%%%%%%%%%%%%%%%%%%%%%%%%%%%%%%%%%%%%%%%%%%%%%%%%%%%%%%%%%%%%%%%%%%%%
\begin{proof}[Proof of Theorem \ref{thm: main}]
Firstly we show that $\frac{1}{{2}}$ is the minimum coefficient. 

Assume $C<\frac{1}{{2}}$, then  
\beq\label{ineq1: fundamental ineq}
\sum_{j=1}^\infty |a_j|^p \preceq \int_{1}^{\infty} C^{2p} j^{-p/2} \; dj= \frac{2C^{2p}}{p-2} \quad \mbox{for}\ p>2. 
\eeq
We notice that $2C^{2p} < 1/2$ for $p=2+\varepsilon$ with small $\varepsilon>0$.  

On the other hand, it follows from Proposition \ref{prop: tau fanction} that 
the sum of the $p$--th power equi-partitions is 
\beq\label{ineq2: fundamental ineq}
\sum_{j=1}^\infty |a_j|^p = \tau(p) \geq (1-2^{1-p})\zeta(p-1) \quad \mbox{for}\ p>2. 
\eeq
We then recall the property of Riemann zeta function $\zeta(x)$ (See e.g. \cite{Ale}):  
\beq
\lim_{p \rightarrow 2+0} \left(\zeta(p-1) - \frac{1}{p-2}\right) = \gamma 
\eeq
where $\gamma$ is Euler's constant. So we have   
\beq
\lim_{p \rightarrow 2+0} \left( \left( 1-2^{1-p} \right) \zeta(p-1) - \frac{1}{2(p-2)}\right) = C
\eeq
for some constant $C (= \frac{1}{2}(\log 2 +\gamma) )$. Thus it follows from \eqref{ineq2: fundamental ineq} that 
$$\lim_{p\rightarrow 2+0} \left(\sum_{j=1}^\infty |a_j|^p -\frac{1}{2(p-1)}\right)  \geq \lim_{p\rightarrow 2+0}\left( \left(1-2^{1-p}\right) \zeta(p-1)  -\frac{1}{2(p-1)}\right)  =C $$
whereas from \eqref{ineq1: fundamental ineq}
$$\lim_{p\rightarrow 2+0}\sum_{j=1}^\infty |a_j|^p -\frac{1}{2(p-1)} = -\infty.$$
This is a contradiction as desired. 

To prove the existence of suitable partitions satisfying $a_j \sim C j^{-1/2}$ for $C\geq 1/2$, we use the spectral properties of the Neumann--Poincar\'e (NP) operator, which is known as boundary integral operartors, defined on boundaries of a region in $\Rbb^3$ (See e.g. \cite{AKPM} and references therein for details). The NP operators on $L^2(\p\GO)$ are compact if $\p\GO$ is in $C^{1, \alpha}$, that is, the corresponding {\it non-zero} spectrum consists of eigenvalues only. 
We emphasize that corresponding eigenvalues on prolate ellipsoids $\p\GO$ satisfy all properties of lengths for odd partitions of the interval $[0, 1/2]$ (See \cite{AA, Ma, Ri}):  
$$
\begin{array}{c}
M_{1, 1}\vspace{1mm} \\
M_{2, 1}\ \hspace{2mm} M_{2, 2}\ \hspace{2mm} M_{2, 3} \vspace{1mm} \\
M_{3, 1}\ \hspace{2mm} M_{3, 2}\ \hspace{2mm} M_{3, 3}\ \hspace{2mm} M_{3, 4}\ \hspace{2mm} M_{3, 5}  \vspace{1mm} \\
\vdots \qquad \vdots \qquad \vdots
\end{array}
$$
Here each column shows a partition of $[0, 1/2]$ and each column consists of an odd number of {\it non-zero} subintervals. The sum $\sum_{k=1}^{2N+1}|M_{N, k}|=1/2$ for each $N \in \Nbb_{\geq 0}$. 
These astonishing facts are not elementary but the results are available here.  
Furthermore, it is recently proven \cite{Mi,MR} that NP eigenvalues satisfy the so-called Wely's law, namely, 
$$
a_j \sim \widetilde{C} j^{-1/2}
$$
for $\widetilde{C} \geq 1/4$. Here the coefficient $\widetilde{C}$ is explicitly calculated by using the Willmore energy $W(\p\GO)$ and the Euler characteristic $\chi(\p\GO)$ of the bouncary surface $\p\GO$. 
It follows that the coefficient $\widetilde{C}$ can take arbitrary real values larger than $1/4$ (See \cite{Mi,MR} for the details). 
As a result, there exist odd partitions of a half interval $[0, 1/2]$ such that the enumerated sequence satisfies 
$$
a_j \sim \widetilde{C} j^{-1/2}
$$
for all $\widetilde{C} \geq 1/4$. When we consider the interval $[0, 1]$ instead of the half interval $[0, 1/2]$, automatically $C=2 \widetilde{C} \geq 1/2$. 
\end{proof}

%\begin{cor}
%For all $c\geq 1$, there exist odd partitions of the unit interval such that 
%$$
%\lim_{N\rightarrow \infty} (2N+1)^{p-1} \sum_{k=0}^{2N+1} |{L_{N, k}}|^p = c \quad p>2.   
%$$
%\end{cor}

\section{Discussions}
We proved the decay rates on odd partitions of the unit interval. Theorem \ref{thm: main} is proven 
by a fundamental property of Riemann zeta function and by the spectral theory of Neumann--Poincar\'e operators. 
If the partitions are permitted to have {\it zero}--length sets, this fact is  proved in an elementary manner. 
Can one give an elementary proof of Theorem \ref{thm: main} as it is? 
To the best of my knowledge, we don't know alternative proofs other than spectral theory. 

The partitions of the unit interval have been considered for a number of years from various viewpoints (See e.g. \cite{DV,HW}).  
For more general partitions, can one prove the existence of decay sequence for suitable orders? 

The difference sequence of Farey sequence, for instance, is the partitions of the unit interval \cite{HW}:  
$$
\begin{array}{c}
1\vspace{1mm} \\
1/2\ \hspace{2mm} 1/2 \vspace{1mm} \\
1/3\ \hspace{2mm} 1/6\ \hspace{2mm} 1/6\ \hspace{2mm} 1/3\vspace{1mm} \\
1/4\ \hspace{2mm} 1/12\ \hspace{2mm} 1/6\ \hspace{2mm} 1/6\ \hspace{2mm} 1/12\ \hspace{2mm} 1/4 \vspace{1mm} \\
1/5\ \hspace{2mm} 1/20\ \hspace{2mm} 1/12\ \hspace{2mm} 1/15\ \hspace{2mm} 1/10\ \hspace{2mm} 1/15\ \hspace{2mm} 1/12\ 
\hspace{2mm} 1/20\ \hspace{2mm} 1/5 \vspace{1mm} \\
\vdots \qquad \vdots \qquad \vdots
\end{array}
$$
Thus the enumerated decreasing sequence $\{a_j\}$ is denoted as 
\beq 1, 1/2, 1/2, 1/3, 1/3, 1/4, 1/4, 1/5, 1/5, 1/6, 1/6, 1/6, 1/6, 1/6, 1/6,  \ldots \eeq 
and its corresponding behavior seems  $a_{j} \sim C j^{-1/3}$, since it is known that the $N$-th column consists of $N^2$ numbers asymptotically. The decreasing order $j^{-1/3}$ depends on the number of partitions. For general partions of the unit interval, 
can one prove the analogous results to Theorem \ref{thm: main}? The minimum decay can be easily guessed and it should be given by equi-partitions.

%%%%%%%%%%%%%%%%%%%%%%%%%%%%%%%%%%%%%%%%%%%%%%%%%%%%%%%%%%%%%%%%%%%%%%%%%%

\end{document}